\documentclass[12pt]{article}
\usepackage[T1]{fontenc}
\usepackage{authblk} 
\usepackage{amsmath}
\usepackage{amsfonts}
\usepackage{amssymb}
\usepackage[normalem]{ulem} 
\usepackage[pagewise]{lineno}
\usepackage{amsmath,xcolor,ulem}
\usepackage{amsfonts}
\usepackage{amssymb}
\usepackage{amsthm}
\usepackage{graphicx}
\usepackage{marginnote}
\usepackage{subcaption}

\usepackage[colorlinks=true, allcolors=blue]{hyperref}
\usepackage[top=2.5cm,bottom=2.5cm,left=2.5cm,right=2.5cm]{geometry}
\usepackage{float}

\usepackage[algoruled]{algorithm2e}
\usepackage[font=footnotesize,width=\linewidth]{caption} 
\usepackage{graphicx}
\usepackage{amssymb}
\usepackage{amsthm}
\usepackage{amsmath}
\usepackage{lineno}
\usepackage{hyperref}
\usepackage{verbatim}

\newtheorem{theorem}{Theorem}[section]
\newtheorem{lemma}[theorem]{Lemma}
\theoremstyle{definition}

\theoremstyle{remark}
\newtheorem{remark}[theorem]{Remark}

\numberwithin{equation}{section}

\usepackage{graphicx, pdflscape}
\usepackage{amsmath, hyperref}
 \usepackage{amssymb} 
\usepackage{relsize, float, lineno}

 \usepackage{xcolor} 
\usepackage{subcaption}
\allowdisplaybreaks

\usepackage[figcolor=white]{todonotes}

\title{On a Generalized Compartment Model for\\ Ethanol Metabolism in the Human Body}
\author[1]{Manh Tuan Hoang\footnote{\href{mailto:tuanhm16@fe.edu.vn}{tuanhm16@fe.edu.vn(corresponding author)}}}
\author[2]{Thi Kim Quy Ngo\footnote{ \href{mailto:quyntk@ptit.edu.vn}{quyntk@ptit.edu.vn}}}
\author[3]{Benjamin Wacker\footnote{ \href{mailto:benjamin.wacker@hs-merseburg.de}{benjamin.wacker@hs-merseburg.de}}}
\affil[1]{Department of Mathematics, FPT University, Hoa Lac Hi-Tech Park, Km29 Thang Long Blvd, Hanoi, Viet Nam}
\affil[2]{Posts and Telecommunications Institute of Technology, Km10 Nguyen Trai, Ha Dong, Hanoi, Viet Nam}
\affil[3]{Department of Engineering and Natural Sciences, University of Applied Sciences Merseburg, Eberhard-Leibnitz-Str. 2, D-06217 Merseburg, Germany}
\begin{document}
\maketitle
\begin{abstract}
We introduce a generalized continuous-time compartment model of ethanol metabolism in the human body that extends a recently developed framework. In the proposed model, we replace the Michaelis-Menten mechanism of the liver's ethanol metabolism rate with a general class of nonlinear rate functions. This modification provides greater modeling flexibility and enables the model to capture a wider range of hepatic ethanol metabolism dynamics. The qualitative behavior of the proposed ethanol metabolism model is analyzed rigorously. More specifically, we investigate the positivity and boundedness of solutions, as well as the global asymptotic stability (GAS) of the unique equilibrium point using an appropriate quadratic Lyapunov function.

Second, we formulate a discrete-time counterpart of the proposed continuous-time model and investigate its dynamical properties. We show that, under an appropriate condition on the time step size, the discrete-time model faithfully reproduces the qualitative dynamical behavior of the corresponding continuous-time system.

Lastly, we conduct a series of numerical experiments employing several ethanol metabolism rate functions to support the theoretical results.
\end{abstract}
\begin{minipage}{0.9\linewidth}
 \footnotesize
\textbf{AMS classification:} 34C60, 39A60 \\
\textbf{Keywords:} 
Ethanol metabolism, Compartment modeling, Nonlinear rate, Dynamical analysis, Global asymptotic stability
\end{minipage}

\section{Introduction}\label{intro}
In very recent works \cite{Wacker1, Wacker2}, Wacker conducted rigorous and extensive studies on mathematical modeling, analysis and numerical simulation of a new ethanol metabolism in the human body. Specially, the following compartment model was proposed in \cite{Wacker1} to describe the ethanol metabolism:
\begin{equation}\label{eq:1}
\begin{split}
A'(t) &= -a A(t), \quad A(0) > 0\\
B'(t) &= b C(t) - b B(t) + aA(t),\quad B(0) \geq 0\\
C'(t) &= b B(t)- bC(t)- c\dfrac{C(t)}{d+C(t)}, \quad C(0) \geq 0,
\end{split}
\end{equation}
where
\begin{itemize}
\item  $A(t), B(t)$ and $C(t)$ denote for the amount of ethanol in the gastrointestinal route, the blood system and liver at the time $t$, respectively;
\item $a$ is the constant elimination rate of ethanol from the gastrointestinal route;
\item $b$ is the constant elimination rate of ethanol from the blood system;
\item $c$ is the constant elimination rate of ethanol from the live;
\item $d$ is the constant rate of the alcohol Michaelis-Menten mechanism.
\end{itemize}
A schematic representation of ethanol transport within the human body, modeling by the three-compartment model \eqref{eq:1}, is described in Figure \ref{Fig:1new}. We refer the readers to \cite{Wacker1} for more details of this model. 
\begin{figure}[H]
\centering
\includegraphics[height=6cm,width=16cm]{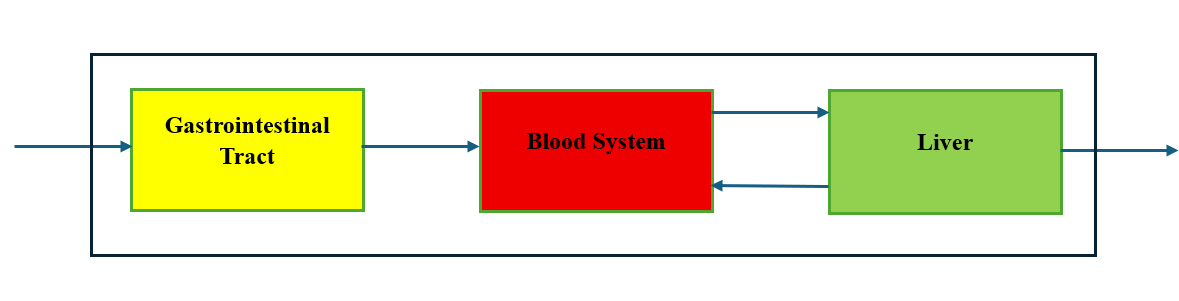}
\caption{The flow of ethanol through the human body.}\label{Fig:1new}
\end{figure}

It is worth noting that \eqref{eq:1} combines two different ethanol metabolism models, which were proposed by Levitt and Levitt in \cite{Levitt} and by Ludwin in \cite{Ludwin}. While Ludwig considers the gastrointestinal tract and the blood system as the model compartments \cite{Ludwin}, Levitt and Levitt use the blood system and the liver, which are considered as the primary sites of ethanol elimination \cite{Levitt}. In contrast, \eqref{eq:1} uses the three compartments to model the flow of ethanol through the human body, consisting of the gastrointestinal tract, the blood system, and the liver, in which ethanol is primarily metabolized by alcohol dehydrogenase \cite{Wacker1}.

In \cite{Wacker1}, a rigorous mathematical study of \eqref{eq:1} was carried out. More clearly, Wacker established the existence and uniqueness of solutions, proved their positivity and boundedness, determined a unique equilibrium state and examined its local asymptotic stability, and constructed a nonstandard finite difference (NSFD) scheme based on Mickens' methodology \cite{Mickens1, Mickens2, Mickens3, Mickens4, Mickens5}.

In \cite{Wacker2}, the global asymptotic stability (GAS) of the equilibrium state of \eqref{eq:1} and two NSFD schemes were investigated. The results demonstrated the validity of the theoretical findings as well as the superiority of the NSFD schemes in long-time numerical simulations \cite{Wacker1, Wacker2}.

We now turn our attention to the function 
\begin{equation}\label{eq:2}
f(C) := c\dfrac{C}{d + C},
\end{equation}
which was used in \eqref{eq:1} to describe the liver's ethanol metabolism rate associated with the Michaelis-Menten mechanism. This function captures the saturation effect \cite{Capasso}, which was commonly used in epidemiological models (see, for instance, \cite{Capasso, Cui, Xu, Zhang} and references therein). Besides the monotone  form \eqref{eq:2}, nonmonotonic functions constitute another suitable choice for $f$ \cite{Ruan, Xiao}. Specifically, $f$ attains its maximum value at a threshold $C= C^*$ and then gradually decreases as $C$ increases further. It can be interpreted as a psychological (inhibitory, overload) effect \cite{Ruan, Xiao}. This effect can be represented by the function \cite{Xiao}
\begin{equation}\label{eq:3}
f(x) = \dfrac{\kappa_1 x}{1 + \kappa_2 x^2}, \quad \kappa_1, \kappa_2 > 0.
\end{equation}
This function attains a maximum value at $x^* = (\sqrt{\kappa_2})^{-1}$. Furthermore, it is increasing when $x < x^*$ and decreasing when $x > x^*$. The graphs of functions $f$ given in \eqref{eq:2} and \eqref{eq:3} are depicted in Figure \ref{Fig:1}.
\begin{figure}[H]
\subfloat[Saturated effect]{%
\includegraphics[height=7.0cm,width=8cm]{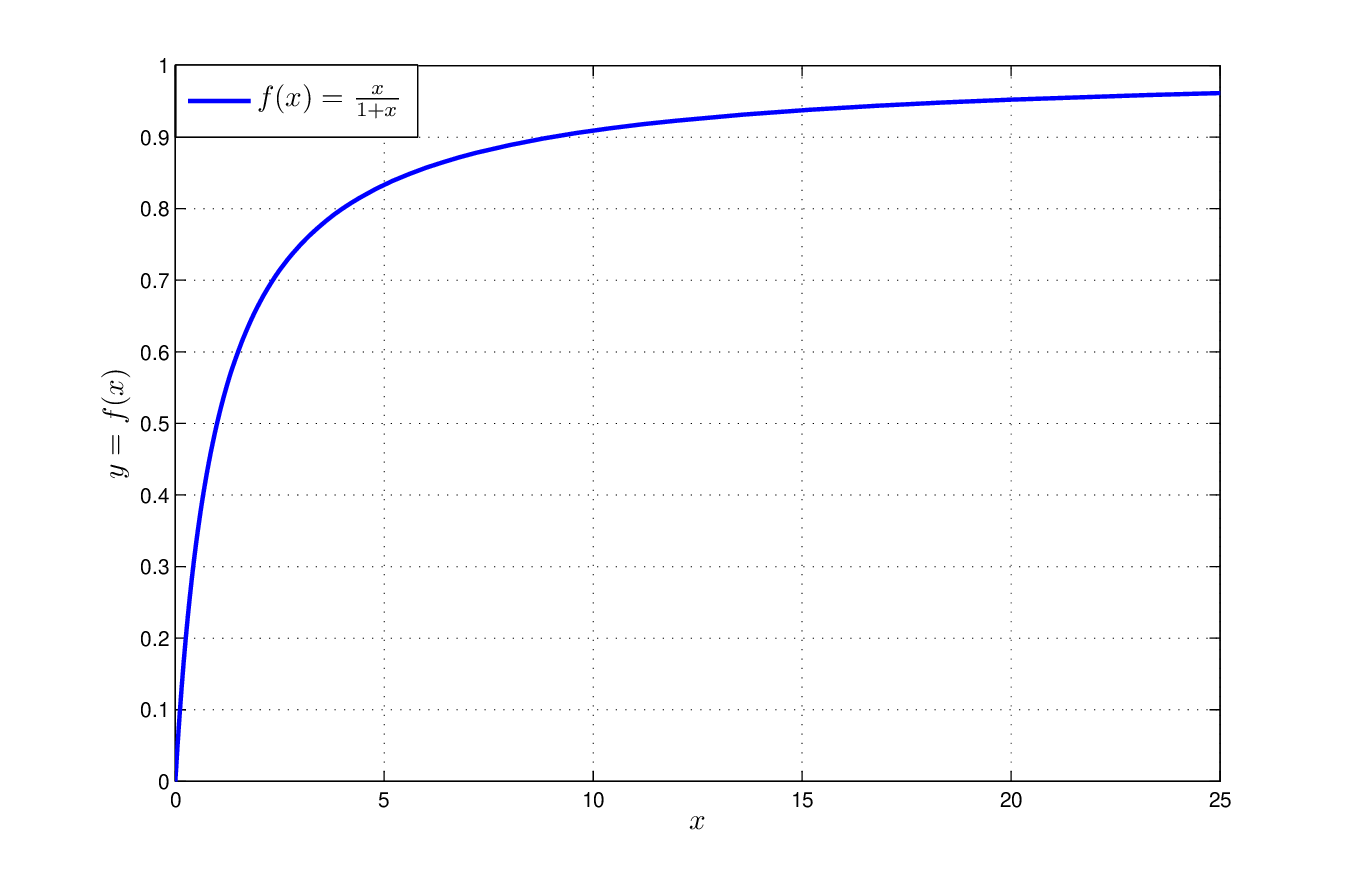}
\label{Figure:1b}
}\hfill
\subfloat[Psychological (inhibitory, overload) effect]{%
\includegraphics[height=7.0cm,width=8cm]{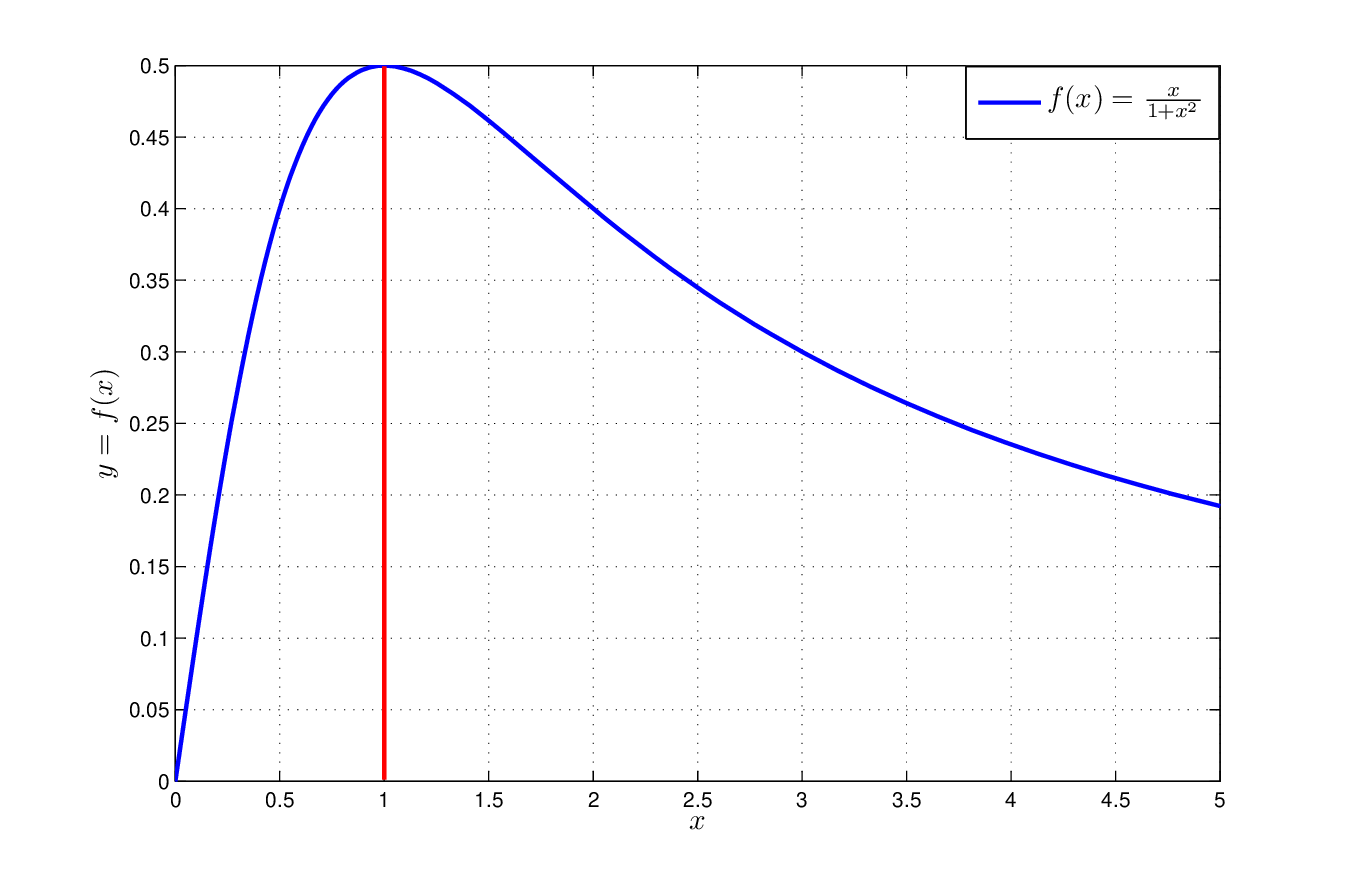}
\label{Figure:1c}
}\hfill
\caption{Monotone and nonmonotone rate functions.}\label{Fig:1}
\end{figure}
Another important class of functional responses in biology is the family of Holling type I, II, III, and IV functions \cite{Dawes1,Wu1}:
\begin{equation*}
\begin{split}
&f_{H, 1} \left( C \right) =  k_{1} \cdot C, \\
&f_{H, 2} \left( C \right) = \dfrac{k_{2, 1} \cdot C}{1 + k_{2, 2} \cdot C}, \\
&f_{H, 3} \left( C \right) = \dfrac{k_{3, 1} \cdot C^{m}}{1 + k_{3, 2} \cdot C^{m}}, \\
&f_{H, 4} \left( C \right) = \dfrac{k_{4, 1} \cdot C^{n}}{1 + k_{4,2} \cdot C^{n + 1}},
\end{split}
\end{equation*}
where $k_{1}, k_{2, 1}, k_{2, 2}, k_{3, 1}, k_{3, 2}, k_{4, 1}, k_{4, 2}$ are positive constants and  $m > 1$, $n \geq 1$. Some classes of general functional responses can be found in \cite{Kalinkat1}.

Inspired by the mentioned-above studies using nonlinear functional rates,  we consider \eqref{eq:1} in a more general framework. Specifically, we propose an extended version represented by

\begin{equation}\label{eq:1new}
\begin{split}
A'(t) &= -a A(t), \quad A(0) > 0,\\
B'(t) &= b C(t) - b B(t) + aA(t),\quad B(0) \geq 0,\\
C'(t) &= b B(t)- b C(t)- f(C(t)), \quad C(0) \geq 0,
\end{split}
\end{equation}
where, without loss of generality, $f$ is assumed to satisfy the following property:\\
\textbf{(P)}: $f(x)$ is continuous function on $[0, \,\,\,\infty)$; $f(x) \geq 0$ for $x \geq 0$ and $f(x) = 0$ if and only if $x = 0$.

In practice, since the ethanol metabolization in the liver depends on a variety of physiological factors, including the availability and activity of alcohol dehydrogenase \cite{Edenberg1}, the liver has a finite capacity to metabolize ethanol. Therefore, it is reasonable to assume that $f$ is bounded above by a positive constant. It is worth noting that the function with the property \textbf{(P)} can be considered as the most general form of the metabolization rate functions. Consequently, the resulting model \eqref{eq:1new} becomes significantly more flexible and general, allowing it to capture a wider range of realistic scenarios by treating $f$ as a control parameter.

In this work, we rigorously analyze the dynamics of the proposed ethanol metabolism model \eqref{eq:1new}. More precisely, we  investigate the positivity and boundedness of the solutions, and global asymptotic stability (GAS) of a unique equilibrium point by an appropriate quadratic Lyapunov function. On the other hand, we conduct a set of numerical experiments, in which several  ethanol metabolism rate functions are used, to support the theoretical findings. The numerical results provide strong evidence supporting the theoretical findings.

{
It is well known that discrete-time dynamical systems have found numerous applications in both theory and practice (see, e.g., \cite{Allen, Allen1, Elaydi, Ladino, Li, May, Wang}). Particularly, discrete-time models possess several advantages over their continuous-time counterparts since in many real-world situations, data are collected at discrete observation times (e.g., daily or hourly), making discrete-time models more suitable for data acquisition and analysis. As pointed out by May in his influential work \cite{May}, discrete-time models arising in real-world applications, although simple and deterministic, may exhibit highly complex and even unpredictable dynamical behaviors that are absent in their continuous-time counterparts. Motivated by both the practical importance and the rich dynamical features of discrete-time systems, we consider a discrete-time version of \eqref{eq:1new} in the form:
\begin{equation}\label{eq:DE1}
\begin{split}
A_{n + 1} &= A_n  -a\Delta t A_n, \quad A_0 > 0,\\
B_{n+1} &= B_n + b\Delta t C_n - b\Delta t B_n + a\Delta t A_n,\quad B_0 \geq 0,\\
C_{n + 1} &= C_n +  b\Delta t B_n- b\Delta t C_n- \Delta tf(C_n), \quad C_0 \geq 0,
\end{split}
\end{equation}
where $n$ represents the time $n\Delta t$ with $\Delta t$ being the step size. Thus, $A_n$, $B_n$ and $C_n$ correspond to the amount of the ethanol in the gastrointestinal route, the blood system and liver at the time $n\Delta t$, respectively.

In Section \ref{Sec2new}, we analyze the dynamics of the discrete-time model \eqref{eq:DE1}. Based on rigorous mathematical analysis, we establish the positivity of the solutions, determine the set of equilibrium points and establish their asymptotic stability properties. The obtained results show that, under a simple condition on the time step size $\Delta t$, the discrete-time model faithfully reproduces the dynamical behavior of its continuous-time counterpart. This property is particularly desirable in practice.
}

This work is organized as follows:\\
Section \ref{Sec2} and \ref{Sec2new} investigate dynamical properties of the proposed models \eqref{eq:1new} and \eqref{eq:DE1}. Numerical experiments are conducted and reported in Section \ref{Sec3}. The last section presents some concluding remarks and discussions.
\section{Dynamical Analysis of the Generalized Ethanol Metabolism Model}\label{Sec2}
This section analyzes the global dynamics of the generalized model \eqref{eq:1new}. First, the existence and uniqueness of solutions can be done by the arguments performed in \cite{Wacker1}. Here, we establish the positivity of the solutions based on a simpler approach than that employed in \cite{Wacker1}.
\begin{lemma}\label{Lemma1}
The model \eqref{eq:1new} admits the following set as a positively invariant set
\begin{equation}\label{eq:5}
\Omega = \big\{(A,\,B,\,C) \in \mathbb{R}^3|A,\, B,\, C \geq 0\big\}.
\end{equation}
Furthermore, the functions $A(t)$ and $S(t) = A(t) + B(t) + C(t)$ are nonincreasing function for $t \geq 0$.
\end{lemma}
\begin{proof}
First, it follows from \eqref{eq:1new} that
\begin{equation*}
\begin{split}
&A'|_{A = 0} = 0,\\
&B'|_{B = 0} = b C + a A \geq 0,\\
&C'|_{C = 0} = b B \geq 0,
\end{split}
\end{equation*}
for all $A, B, C \geq 0$. Therefore, we deduce from \cite[Proposition B.7]{Smith} that $A(t), B(t), C(t) \ge0$ for $t>0$, whenever $A(0), B(0), C(0) \ge0$.

From the first equation of \eqref{eq:1new}, we have $A(t)$ is a nonincreasing function.

On the other hand, it follows from the all equations of \eqref{eq:1new} that
\begin{equation*}
A'(t) + B'(t) + C'(t) = -f(C(t)) \leq 0,
\end{equation*}
which means that $A(t) + B(t) + C(t)$ is nonincreasing function for $t > 0$.

In conclusion, $\Omega$ forms a positively invariant set of \eqref{eq:1new}. The proof is complete.
\end{proof}

Next, we determine possible equilibrium points of \eqref{eq:1new}. Any equilibrium point is a solution to the system
\begin{equation}\label{eq:EQ1}
\begin{split}
0&= -a A,\\
0&= b C - b B + aA,\\
0&= b B- b C- f(C).
\end{split}
\end{equation}
From the property \textbf{(P)}, we conclude that \eqref{eq:1new} has only a trivial equilibrium point $E^* = (0,\,0,\,0)$. Its global asymptotic stability is established as follows.
\begin{theorem}\label{Theorem1}
The unique trivial equilibrium point $E^*$ is \eqref{eq:1new} is globally asymptotically stable.
\end{theorem}
\begin{proof}
First, it follows from the first equation of \eqref{eq:1new} that
\begin{equation*}
A(t) = A(0)e^{-at},
\end{equation*}
which implies that $A^* = 0$ is a globally asymptotically stable equilibrium point of the first equation of \eqref{eq:1new}. By using sufficient conditions for the global stabilizability of nonlinear cascade systems \cite{Seibert}, the GAS of \eqref{eq:1new} is reduced to that of the following subsystem
\begin{equation}\label{eq:2new}
\begin{split}
B' &= b C - b B\\
C'&= b B - bC - f(C),
\end{split}
\end{equation}
which is obtained by substituting $A = 0$ into the second equation of \eqref{eq:1new}. We only need to show that $(B^*,\,C^*) = (0,\,0)$ is a globally asymptotically stable equilibrium point of \eqref{eq:2new}. Indeed, consider a Lyapunov function candidate defined by
\begin{equation}\label{eq:6}
V(B, C) = \dfrac{1}{2}B^2 + \dfrac{1}{2}C^2.
\end{equation}
The time derivative of $V$ along with the solutions of \eqref{eq:2new} is computed as
\begin{equation*}
\begin{split}
\dfrac{dV}{dt} &= \dfrac{dV}{dB}\dfrac{dB}{dt} + \dfrac{dV}{dC}\dfrac{dC}{dt}\\
&=B(bC - bB) + C(b B - bC - f(C)) = - b(B - C)^2 - Cf(C).
\end{split}
\end{equation*}
Thus, $\frac{dV}{dt} \leq 0$ for all $B, C \geq 0$ and $\frac{dV}{dt} = 0$ if and only if $B = C = 0$. Consequently, the Lyapunov stability theorem \cite{Khalil} implies the GAS of $(B^*,\,C^*) = (0,\,0)$.

By using the global stabilizability results of two cascade-connected nonlinear systems \cite[Corollary 4.3]{Seibert}, we conclude that $E^*$ of \eqref{eq:1new} is globally asymptotically stable. The proof is complete.
\end{proof}
\section{Dynamical analysis of the discrete-time model}\label{Sec2new}
{
This section is devoted to analyzing dynamical properties of the discrete-time model \eqref{eq:DE1}. Throughout this section, we additionally assume that the function $f$ satisfies:\\
\textbf{({P*}):} The derivative of $f(C)$ at zero exists.\\
Let us introduce the function
\begin{equation}\label{eq:DE0}
g(C) =
\begin{cases}
f'(0), \quad \mbox{if} \quad C = 0,\\
&\\
\dfrac{f(C)}{C} \quad \mbox{if} \quad C \ne 0.
\end{cases}
\end{equation}
Then, $f(C) = Cg(C)$ for $C \geq 0$ and $g(C)$ is continuous. For arbitrary fixed positive values of $A_0, B_0$ and $C_0$, we set $S_0 = A_0 + B_0 + C_0$ and define
\begin{equation}\label{eq:DE3}
M = \max_{0 \leq C \leq S_0}g(C).
\end{equation}
The biological relevance of \eqref{eq:DE1} requires certain conditions on the parameter values. The following result provides conditions under which \eqref{eq:DE1} is biologically meaningful.
\begin{lemma}[Biological relevance]\label{Lemma1}
The system \eqref{eq:DE1} is biologically meaningful under the assumption that
\begin{equation}\label{eq:DE2}
\Delta t \leq \min\bigg\{\dfrac{1}{a},\,\,\,\dfrac{1}{b + M}\bigg\}.
\end{equation}
Furthermore, under this condition, the sequences $\{A_n\}$ and $\{A_n + B_n + C_n\}$ ($n \geq 0$) are nonincreasing.
\end{lemma}
\begin{proof}
First, we rewrite \eqref{eq:DE1} in the form
\begin{equation}\label{eq:DE7}
\begin{split}
A_{n + 1} &= (1 -a\Delta t) A_n,\\
B_{n+1} &= (1 - b\Delta t)B_n + b\Delta t C_n + a\Delta t A_n,\\
C_{n + 1} &= \big(1 - b\Delta t - g(C_n)\Delta t\big)C_n +  b\Delta t B_n.
\end{split}
\end{equation}
Thus, we have
\begin{equation*}
\begin{split}
A_{1} &= (1 -a\Delta t) A_0,\\
B_{1} &= (1 - b\Delta t)B_0 + b\Delta t C_0 + a\Delta t A_0,\\
C_{1} &= \big(1 - b\Delta t - g(C_0)\Delta t\big)C_0 +  b\Delta t B_0.
\end{split}
\end{equation*}
On the other hand, it follows from \eqref{eq:DE2} that
\begin{equation*}
\begin{split}
&1 - a\Delta t \geq 0,\\
&1 - b \Delta t \geq 1 - b \Delta t - g(C_0)\Delta t \geq 1 - b\Delta t - M\Delta t \geq 0.
\end{split}
\end{equation*}
Thus, we conclude that: $A_1,\,B_1,\,C_1 \geq 0$.

From \eqref{eq:DE1}, we obtain
\begin{equation*}
\begin{split}
&A_{1} - A_0 = -a\Delta t \leq 0,\\
&\big(A_{1} + B_{1} + C_{1}\big) - \big(A_0 + B_0 + C_0\big) = -f(C_0)\Delta t \leq 0,
\end{split}
\end{equation*}
which implies that $A_1 \leq A_0$ and $A_1 + B_1 + C_1 \leq A_0 + B_0 + C_0$.

Repeating the above arguments, we obtain: If $A_0, B_0, C_0 \geq 0$, then $A_n,\,B_n,\,C_n \geq 0$, $A_n \leq A_{n - 1}$ and $A_n + B_n + C_n \leq A_{n - 1} + B_{n - 1} + C_{n - 1}$ for $n \geq 1$. This is the desired conclusion. The proof is complete.
\end{proof}
Next, we determine the set of equilibrium points of \eqref{eq:DE1}. It is easy to verify that any equilibrium point of \eqref{eq:DE1} is a solution of the system \eqref{eq:EQ1}. Hence, similarly to the continuous-time model \eqref{eq:1new}, \eqref{eq:DE1} also possesses a unique equilibrium point $E^* = (0,\,0,\,0)$. Its asymptotic stability is established as follows.
\begin{theorem}[Local asymptotic stability analysis]\label{Theorem5}
Under the condition \eqref{eq:DE2}, the trivial equilibrium point of \eqref{eq:DE1} is locally asymptotically stable for all parameter values.
\end{theorem}
\begin{proof}
The Jacobian matrix of the system \eqref{eq:DE1} evaluated at $E^*$ is given by
\begin{equation}\label{LAS}
J := J(E^*) = 
\begin{pmatrix}
1 - a\Delta t&0&0\\
a\Delta t&1 - b\Delta t& b\Delta t\\
0&b\Delta t & 1 - b\Delta t - f'(0)\Delta t
\end{pmatrix}.
\end{equation}
Thus, $J$ has three eigenvalues $\lambda_i$ ($i = 1, 2, 3$), where $\lambda_1 = 1 - a\Delta t$ and $\lambda_2$ and $\lambda_3$ are the eigenvalues of the submatrix
\begin{equation*}
J_1 = 
\begin{pmatrix}
1 - b\Delta t& b\Delta t\\
b\Delta t & 1 - b\Delta t - f'(0)\Delta t
\end{pmatrix}.
\end{equation*}
It is easy to see that $|\lambda_1| < 1$. On the other hand, we have
\begin{equation*}
\begin{split}
Tr(J_1) &= 2 - 2b\Delta t - f'(0)\Delta t,\\
\det(J_1) &= 1 - 2b\Delta t - f'(0)\Delta t + bf'(0)(\Delta t)^2.
\end{split}
\end{equation*}
By using \eqref{eq:DE2}, we obtain
\begin{equation*}
\begin{split}
\det(J_1) - 1 &= -b\Delta t - f'(0)\Delta t - b\Delta t(1 - f'(0)\Delta t) < 0\\
1 - Tr(J_1) + \det(J_1) &=  bf'(0)(\Delta t)^2> 0,\\
1 + Tr(J_1) + \det(J_1) &= 1 - 4b\Delta t - 2f'(0)\Delta t +  bf'(0)(\Delta t)^2\\
&> 4\big(1 - b\Delta t - f'(0)\Delta t\big) + bf'(0)(\Delta t)^2 > 0,
\end{split}
\end{equation*}
which is equivalent to
\begin{equation*}
|Tr(J_1)| < 1 + \det(J_1) < 2.
\end{equation*}
By using \cite[Theorem 2.10]{Allen1}, we conclude that $|\lambda_2| < 1$ and $|\lambda_3| < 1$. Then, the local asymptotic stability of $E^*$ is derived from the linearalized method \cite[Theorem 1.3.7]{Stuart}. The proof is complete.
\end{proof}
In order to analyze the GAS of \eqref{eq:DE1}, we need the following assumption: The function $g$ defined in \eqref{eq:DE0} attains a positive minimum value on $[0,\,\,\,S_0]$, that is 
\begin{equation}\label{eq:DE3a}\
\min_{0 \leq C \leq S_0}g(C) = m > 0,
\end{equation}
for some positive real number $m$.

\begin{theorem}[Global asymptotic stability analysis]\label{Theorem5new}
The trivial equilibrium point of \eqref{eq:DE1} is not only locally asymptotically stable but also globally asymptotically stable for all parameter values if \eqref{eq:DE2} and \eqref{eq:DE3a} hold.
\end{theorem}
\begin{proof}
From Theorem \ref{Theorem5}, it is sufficient to show that
\begin{equation*}
\lim_{n \to \infty}\big(A_n,\,B_n,\,C_n\big) = (0,\,0,\,0).
\end{equation*}
It follows from the last equation of \eqref{eq:DE7} that
\begin{equation*}
C_{n + 1} = (1 - b\Delta t - g(C_n)\Delta t)C_n +  b\Delta t B_n \leq (1 - b\Delta t - m\Delta t)C_n +  b\Delta t B_n.
\end{equation*}
Thus, \eqref{eq:DE7} implies that
\begin{equation*}
\begin{pmatrix}
A_{n + 1}\\
B_{n + 1}\\
C_{n + 1}
\end{pmatrix}
\leq
P
\begin{pmatrix}
A_n\\
B_n\\
C_n
\end{pmatrix},
\end{equation*}
where $P$ is defined by
\begin{equation*}
P :=
\begin{pmatrix}
1 - a\Delta t&0&0\\
a\Delta t&1 - b\Delta t& b\Delta t\\
0&b\Delta t & 1 - b\Delta t - m\Delta t.
\end{pmatrix}.
\end{equation*}
From \eqref{eq:DE2}, all the entries of $P$ are positive. Thus, for $n \geq 1$
\begin{equation}\label{eq:DE8}
\begin{pmatrix}
A_n\\
B_n\\
C_n
\end{pmatrix}
\leq P^n
\begin{pmatrix}
A_0\\
B_0\\
C_0
\end{pmatrix}.
\end{equation}
The matrix $P$ always has an eigenvalue given by $\mu_1 = 1 - a\Delta t$, which satisfies $|\mu_1| < 1$. Meanwhile, the two remaining eigenvalues coincide with those of the submatrix
\begin{equation*}
P^* :=
\begin{pmatrix}
1 - b\Delta t& b\Delta t\\
b\Delta t & 1 - b\Delta t - m\Delta t.
\end{pmatrix}.
\end{equation*}
Under the condition \eqref{eq:DE2}, we have the estimates
\begin{equation*}
\begin{split}
\det(P^*) - 1 &= -b\Delta t - m\Delta t - b\Delta t(1 - m\Delta t) < 0\\
1 - Tr(P^*) + \det(P^*) &=  bm(\Delta t)^2> 0,\\
1 + Tr(P^*) + \det(P^*) &= 1 - 4b\Delta t - 2m\Delta t +  bm(\Delta t)^2\\
&> 4\big(1 - b\Delta t - m\Delta t\big) + bm(\Delta t)^2 > 0.
\end{split}
\end{equation*}
Consequently, we conclude that all the eigenvalues $\mu_i$ ($i = 1, 2, 3$) of $P$ satisfy $|\mu_i| < 1$, which implies that $\lim_{n \to \infty}P^n = 0$. Then, it follows from \eqref{eq:DE8} that $\lim_{n \to \infty}\big(A_n,\,B_n,\,C_n\big) = (0,\,0,\,0)$. This is the desired conclusion. The proof is complete.
\end{proof}

\begin{remark}
The conditions \eqref{eq:DE3} and \eqref{eq:DE3a} imply that
\begin{equation*}
mC \leq f(C) \leq MC, \quad C \geq 0.
\end{equation*}
This property is commonly observed in many biological, ecological, and epidemiological models \cite{Allen1, Smith}. The saturating function used in \eqref{eq:1} is a typical example.
\end{remark}
}

{
In the following theorem, we provide another proof of the GAS of the trivial equilibrium of \eqref{eq:DE1}, in which the condition \eqref{eq:DE3a} is no longer required.
}
\begin{theorem}[Global asymptotic stability analysis without condition \eqref{eq:DE3a}]\label{Theorem6new}
The trivial equilibrium point of \eqref{eq:DE1} is not only locally asymptotically stable but also globally asymptotically stable for all parameter values if \eqref{eq:DE2} holds.
\end{theorem}
\begin{proof}
Consider a Lyapunov function candidate defined by
\begin{equation*}
V\big(A_n,\,B_n,\,C_n\big) = \big(A_n + B_n + C_n\big)^2.
\end{equation*}
Then, the variation of $V$ relative to \eqref{eq:DE1} is given by
\begin{equation}\label{eq:GASnew}
\begin{split}
&\Delta V\big(A_n,\,B_n,\,C_n\big)  = V\big(A_{n + 1},\,B_{n + 1},\,C_{n + 1}\big) - V\big(A_n,\,B_n,\,C_n\big) \\
&= \big(A_{n+1} + B_{n+1} + C_{n+1}\big)^2 - \big(A_n + B_n + C_n\big)^2\\
&= \Big[\big(A_{n+1} + B_{n+1} + C_{n+1}\big) - \big(A_n + B_n + C_n\big)\Big] \times \Big[\big(A_{n+1} + B_{n+1} + C_{n+1}\big) + \big(A_n + B_n + C_n\big)\Big]\\
&=-\Delta t f(C_n)\Big[\big(A_{n+1} + B_{n+1} + C_{n+1}\big) + \big(A_n + B_n + C_n\big)\Big],
\end{split}
\end{equation}
which implies that $\Delta V\big(A_n,\,B_n,\,C_n\big)  \leq 0$ for all $A_n, B_n, C_n \geq 0$. From Lyapunov stability theory \cite{Elaydi, LaSalle}, we conclude that the origin is stable.

On the other hand, \eqref{eq:GASnew} implies that $\Delta V = 0$ if and only if $C_n = 0$. Thus, we deduce from LaSalle's invariance principle \cite{Elaydi, LaSalle} that
\begin{equation*}
\lim_{n \to \infty}\big(A_n,\,B_n,\,C_n\big) = (0,\,0,\,0).
\end{equation*}
Consequently, the trivial equilibrium point is globally asymptotically stable. This is the desired conclusion. The proof is completed.
\end{proof}
\section{Numerical Experiments}\label{Sec3}
Here, we conduct numerical experiments to support the theoretical findings. 
\subsection{Dynamics of the continuous-time model}\label{Subsec3.1}
In numerical experiments reported below, we use the classical four-stage Runge-Kutta method (see \cite{Ascher}) with a step size $h = 10^{-4}$ to solve the model \eqref{eq:1new}. Also, the parameters given in Table \ref{Table1} will also be used. The graphs of the rate functions in Table \ref{Table1} are represented in Figure \ref{Fig:2}.
\begin{table}[H]
\centering
\caption{The parameters used in the numerical experiments}\label{Table1}
\begin{tabular}{cccccccccccc}
\hline
Set & $a$ & $b$ & $\alpha$ & $\beta$ & $f(C)$ & Remark\\
\hline 
1 & 0.08 & 0.25 & 0.1 & 1 & $f_1(C) = \dfrac{\alpha C}{1 + \beta C}$ & Saturated (monotone) function\\
\hline
2 & 0.08 & 0.25 & 0.1 & 1 & $f_2(C) = \alpha\big(1 - e^{-\beta C}\big)$ & Saturated (monotone) function\\
\hline
3 & 0.08 & 0.25 & 0.1 & 1 & $f_3(C) = \dfrac{\alpha C}{1 + \beta C^2}$ & Nonmonotone function\\
\hline
4 & 0.08 & 0.25 & 0.1 & 1 & $f_4(C) = \dfrac{\alpha C^2}{1 + \beta C^{3}}$ & Nonmonotone function\\
\hline
\end{tabular}
\end{table}
The solutions of the generalized model \eqref{eq:1new} over $t \in [0,\,200]$, using the parameters in Table \ref{Table1} and the set of initial data $\Gamma := \big\{\big(A_0,\,0,\,0\big)|A_0 = 1, 2, 3, 4, 5\big\}$, are represented in Figures \ref{Fig:3}--\ref{Fig:6}. It is clear that the unique equilibrium point is globally asymptotically stable. However, compared with the other cases, the solutions corresponding to the $f_4$ function approach the stable equilibrium at a slower rate. To see the difference between the solutions, we consider the components $B$ and $C$ over the interval $[0,\,100]$ as shown in Figures \ref{Fig:7} and \ref{Fig:8}. From these, we observe that the different rate functions lead to distinct dynamical behaviors of the corresponding models. Therefore, by treating $f$ as a controllable parameter, the model becomes useful for parameter estimation and identification problems.
\begin{figure}[H]
\centering
\includegraphics[height=10cm,width=16cm]{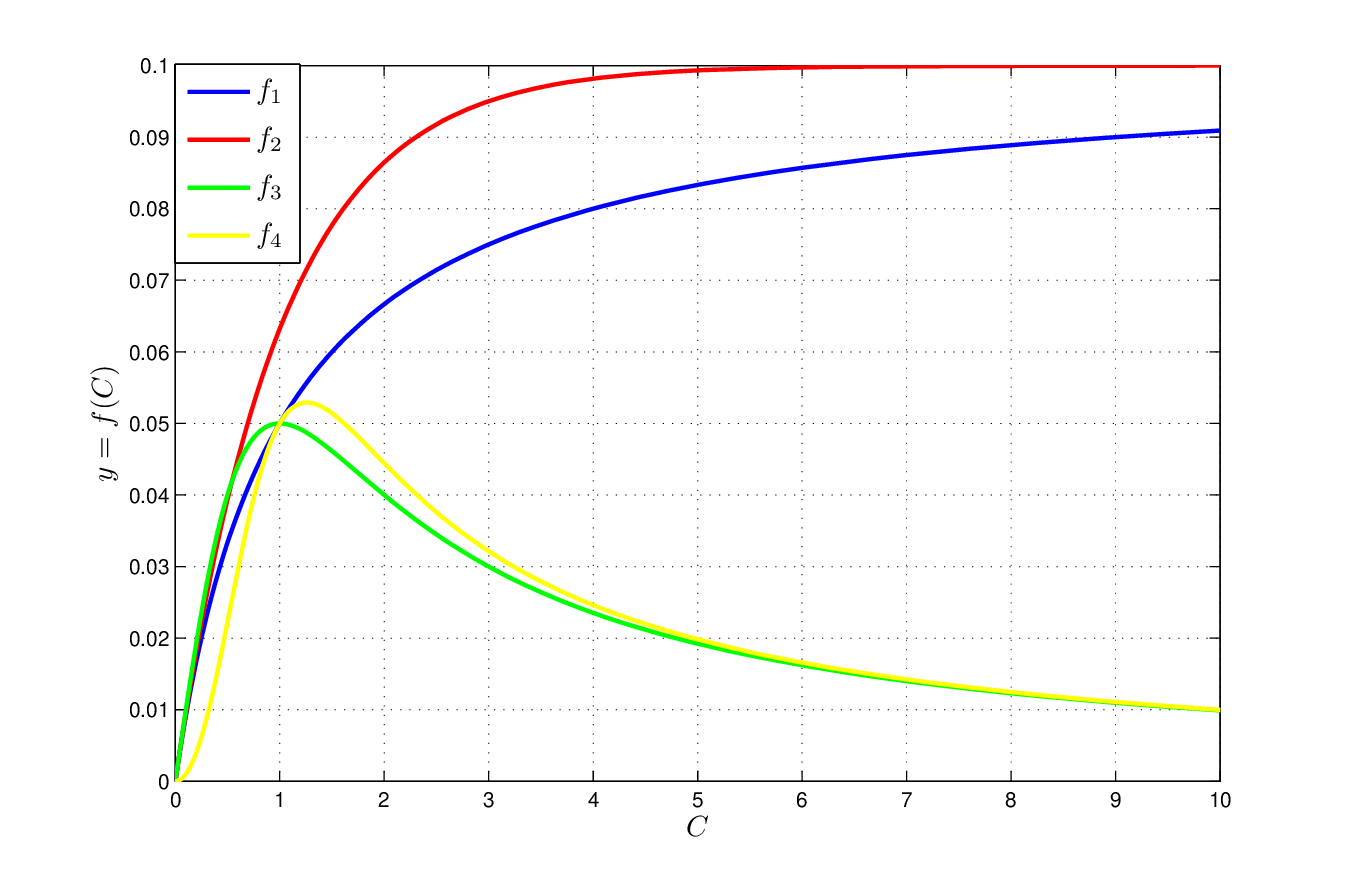}
\caption{The graphs of the rate functions in Table \ref{Table1}.}\label{Fig:2}
\end{figure}
\begin{figure}[H]
\centering
\includegraphics[height=10cm,width=16cm]{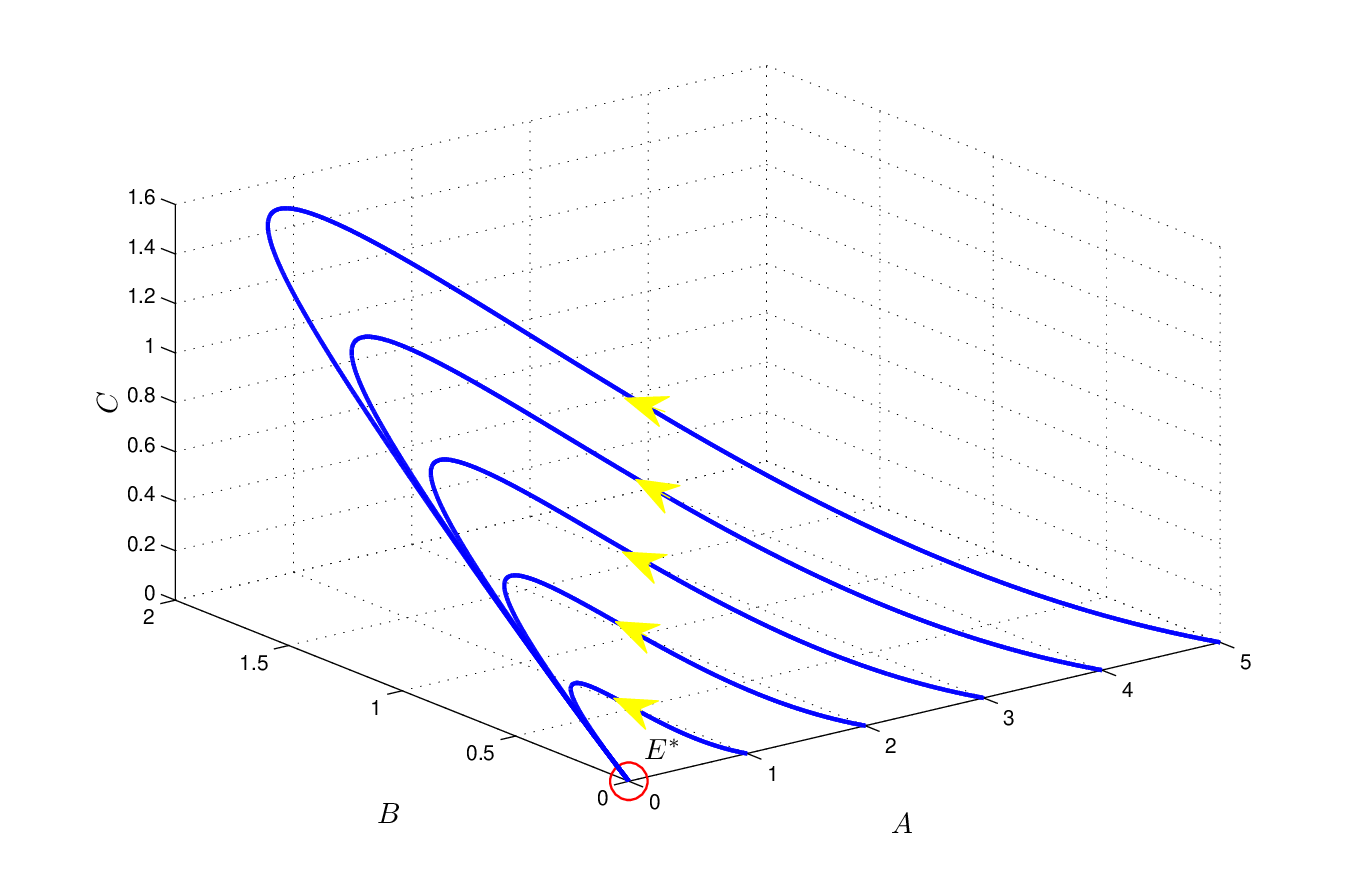}
\caption{The solutions in the phase spaces with the parameter Set $1$ in Table \ref{Table1}.}\label{Fig:3}
\end{figure}
\begin{figure}[H]
\centering
\includegraphics[height=10cm,width=16cm]{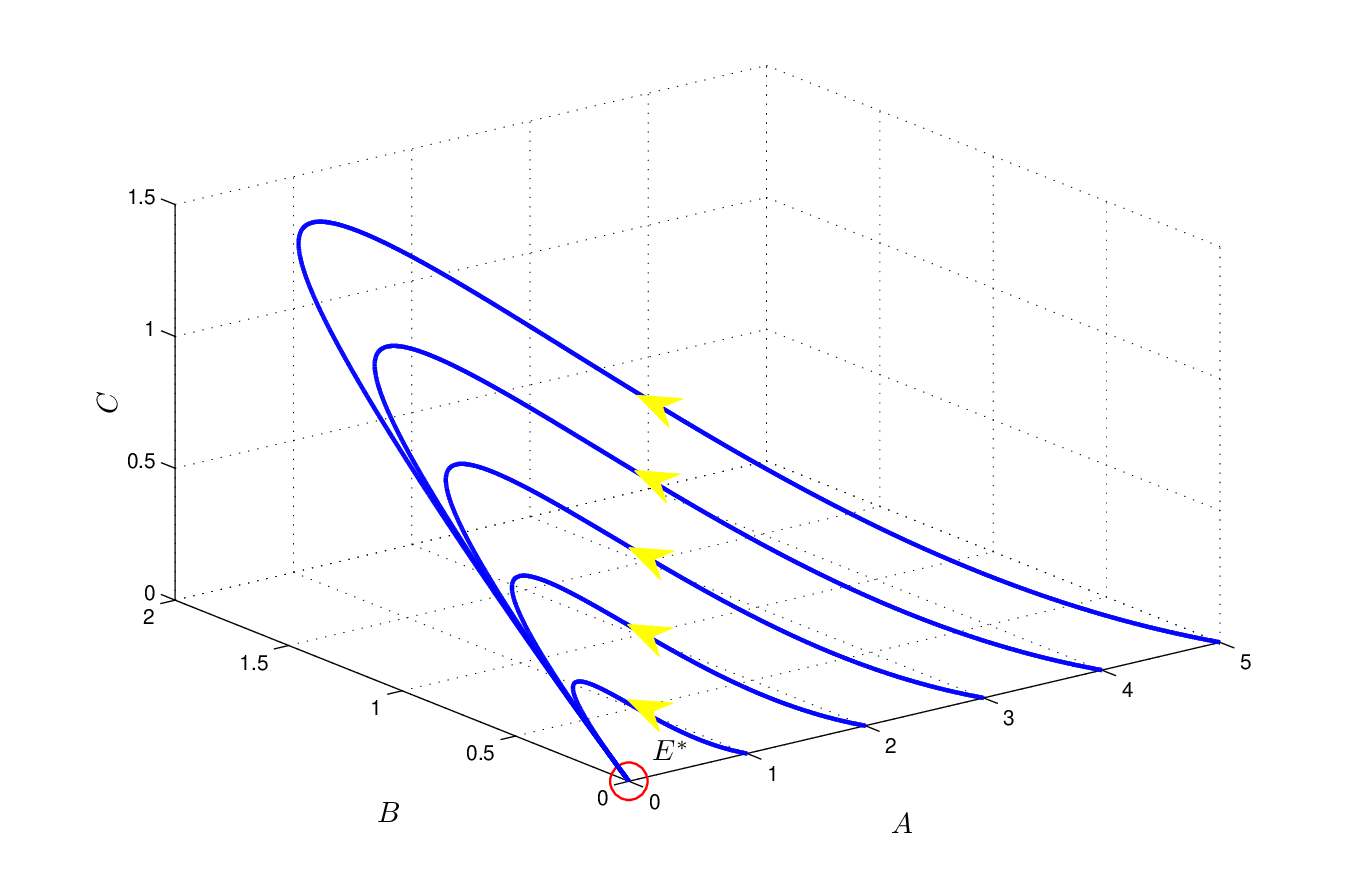}
\caption{The solutions in the phase spaces with the parameter Set $2$ in Table \ref{Table1}.}\label{Fig:4}
\end{figure}
\begin{figure}[H]
\centering
\includegraphics[height=10cm,width=16cm]{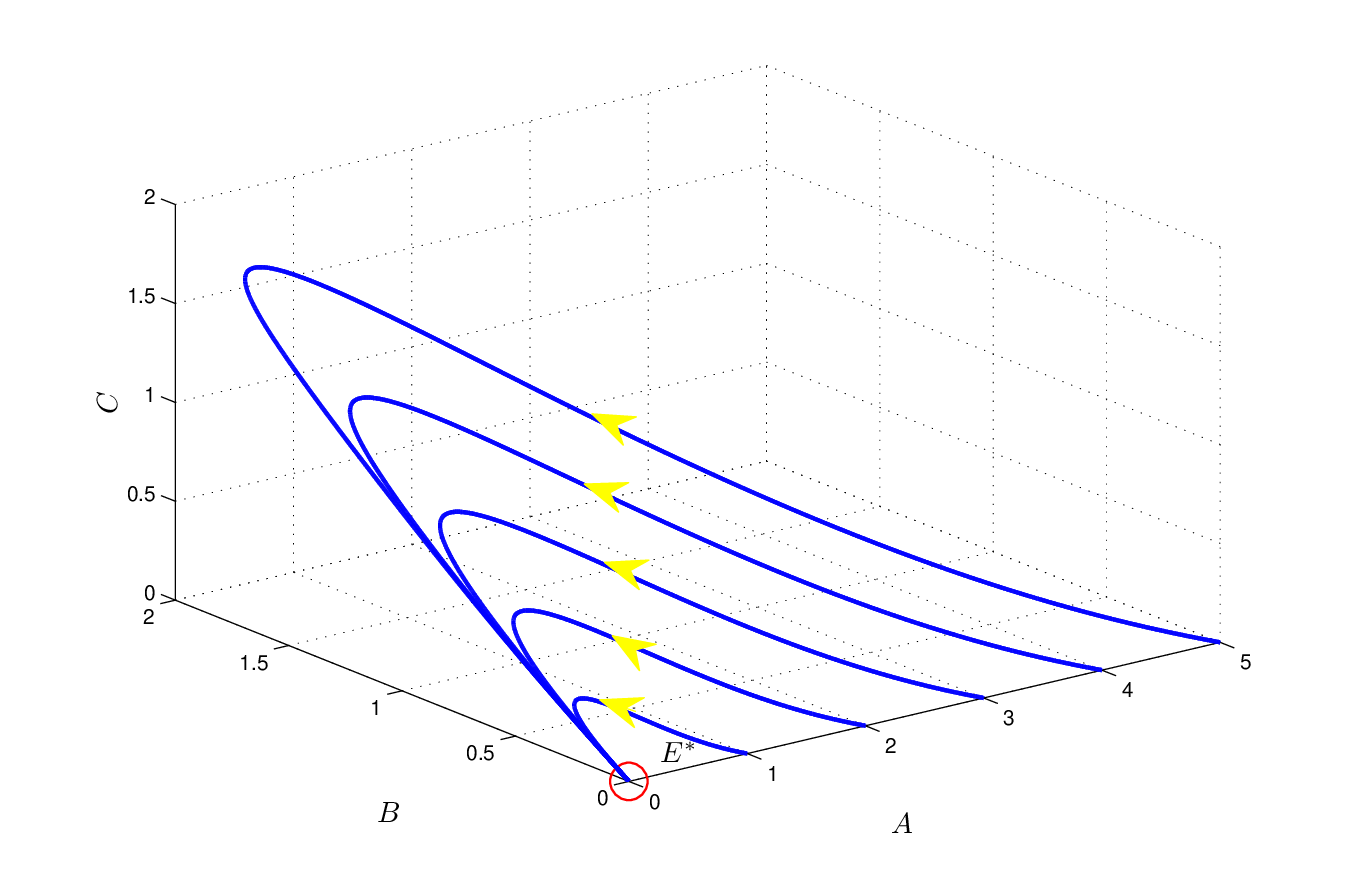}
\caption{The solutions in the phase spaces with the parameter Set $3$ in Table \ref{Table1}.}\label{Fig:5}
\end{figure}
\begin{figure}[H]
\centering
\includegraphics[height=10cm,width=16cm]{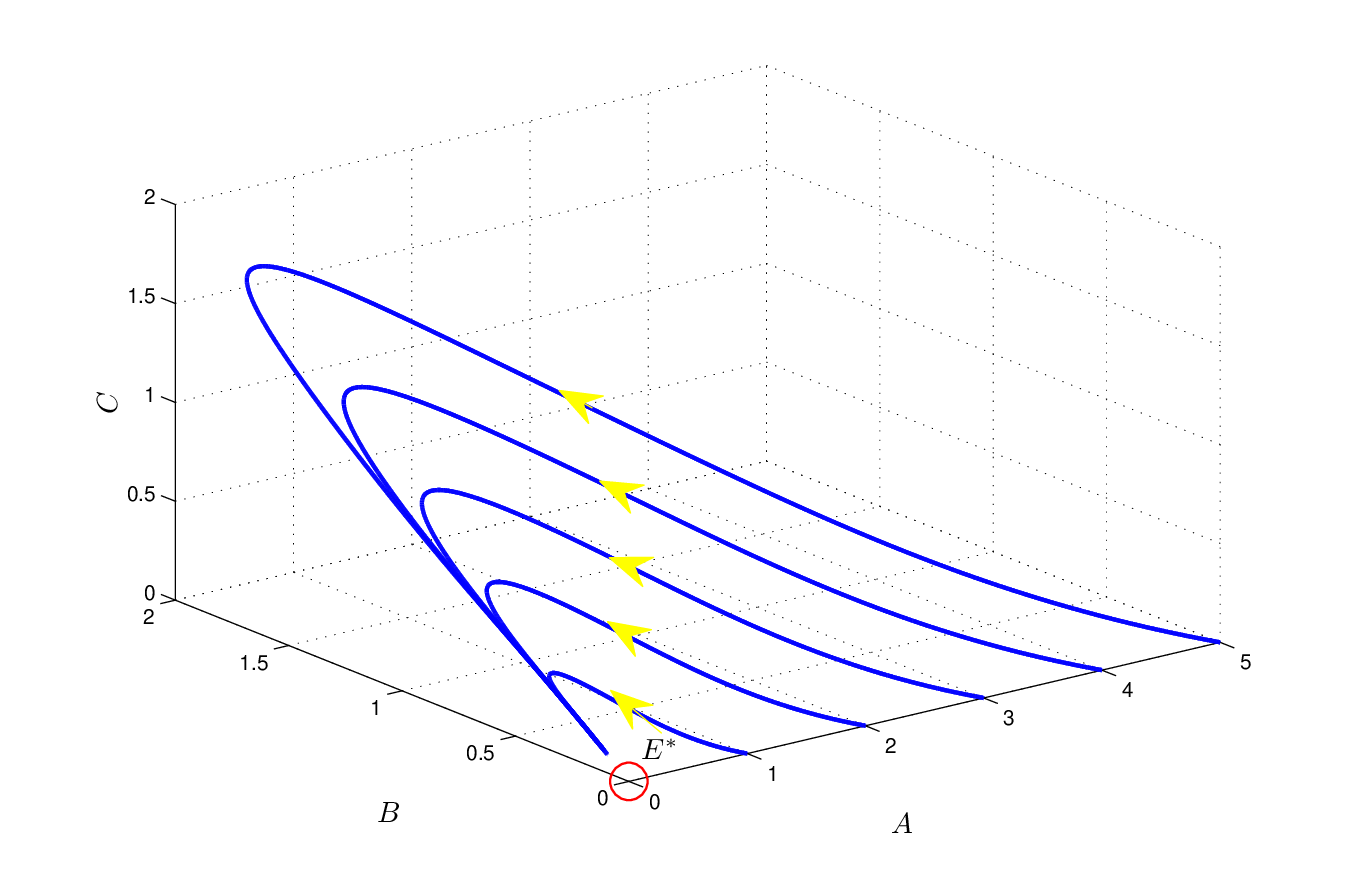}
\caption{The solutions in the phase spaces with the parameter Set $4$ in Table \ref{Table1}.}\label{Fig:6}
\end{figure}
\begin{figure}[H]
\centering
\includegraphics[height=10cm,width=16cm]{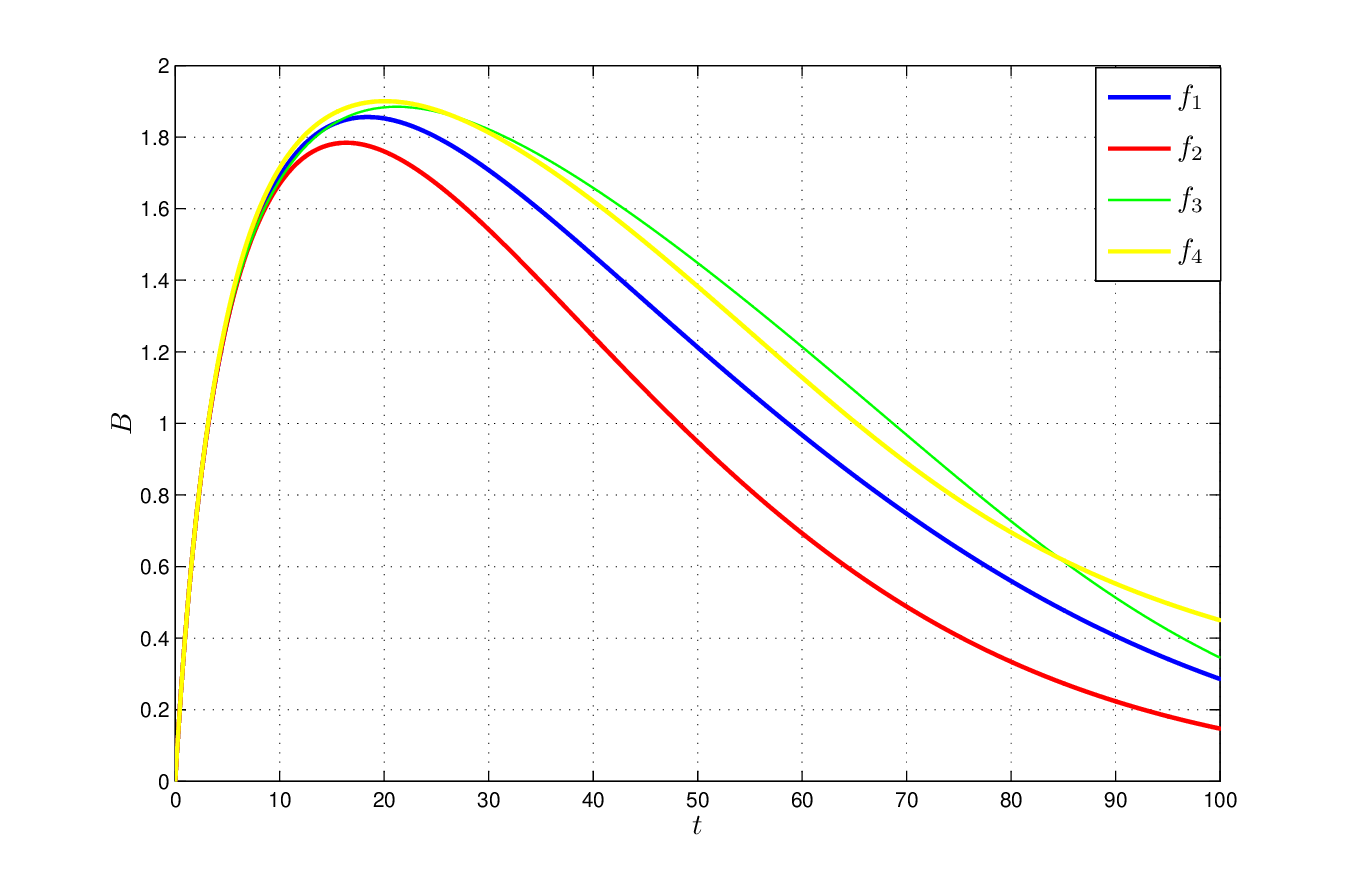}
\caption{The $B$ components corresponding to the functions $f_1, f_2, f_3, f_4$ in Table \ref{Table1}.}\label{Fig:7}
\end{figure}
\begin{figure}[H]
\centering
\includegraphics[height=10cm,width=16cm]{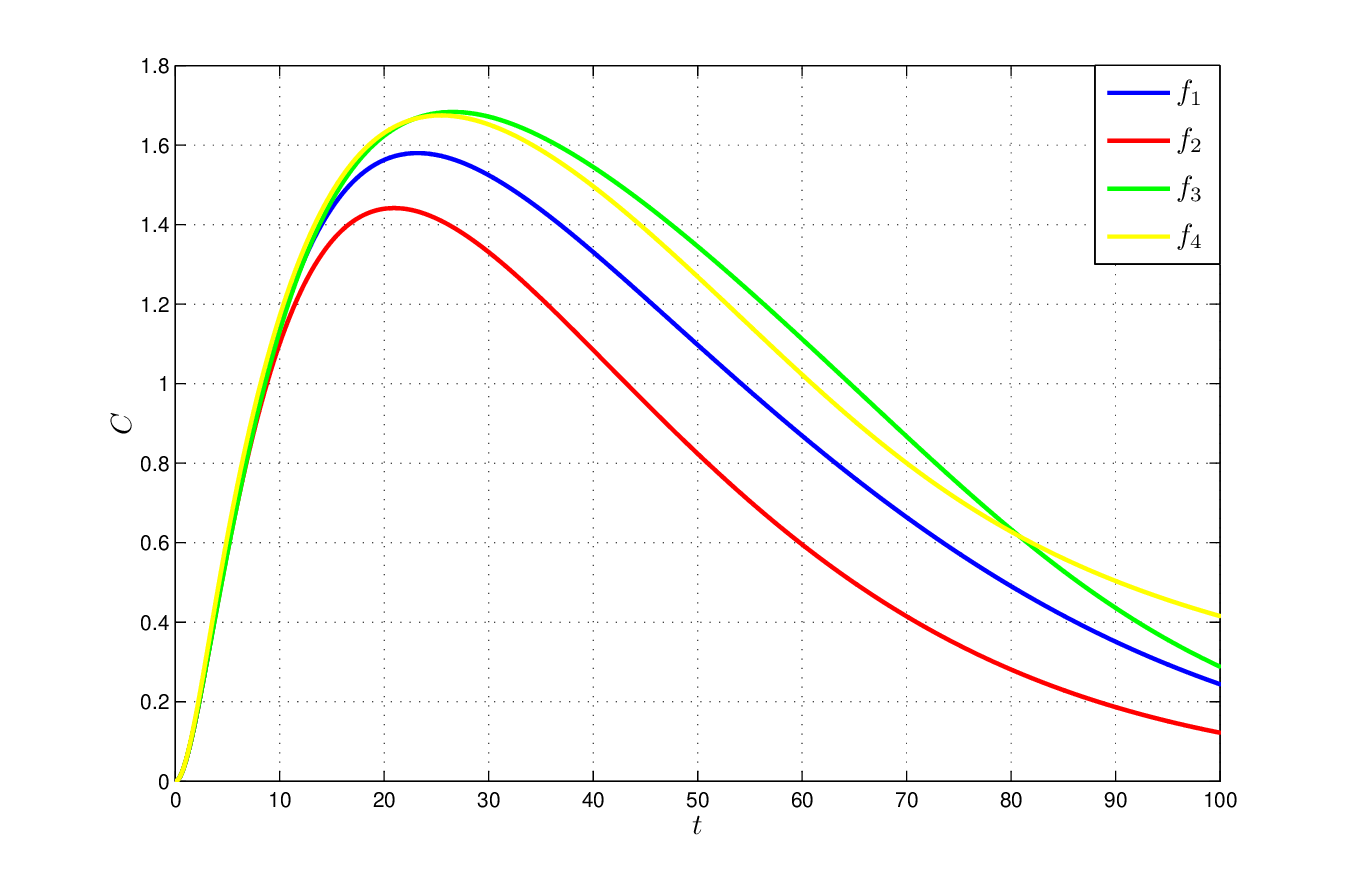}
\caption{The $C$ components corresponding to the functions $f_1, f_2, f_3, f_4$ in Table \ref{Table1}.}\label{Fig:8}
\end{figure}
{
\subsection{Dynamics of the discrete-time model}\label{Subsec3.2}
In this subsection, we examine the dynamics of the discrete-time model \eqref{eq:DE1}. For this purpose, we consider it with the parameters
\begin{equation*}
a = 0.08, \quad b = 0.25,
\end{equation*}
and the function $f(C)$ is given in \eqref{eq:3}, namely:
\begin{equation*}
f(C) = \dfrac{\kappa_1 C}{1 + \kappa_2 C^2}.
\end{equation*}
For this function, we have $g(C)  = \frac{\kappa_1}{1 + \kappa_2 C^2}$ and then,
\begin{equation*}
m = \min_{0 \leq C \leq S_0}g(C) = \dfrac{\kappa_1}{1 + \kappa_2 S_0^2}, \quad M = \max_{0 \leq C \leq S_0}g(C) = {\kappa_1}.
\end{equation*}
In numerical examples reported below, we take  $\kappa_1 = 0.5$ and $\kappa_2 = 1.0$. The solutions of the discrete-time model \eqref{eq:DE1} with some specific initial data and $\Delta t = 1$ are depicted in Figures \ref{Fig:9}-\ref{Fig:12}. It is clear that all the solutions are positive, stable and convergent to the trivial equilibrium point. Thus, the theoretical assertions constructed in Section \ref{Sec3} are illustrated and supported.
\begin{figure}[H]
\centering
\includegraphics[height=9.5cm,width=16cm]{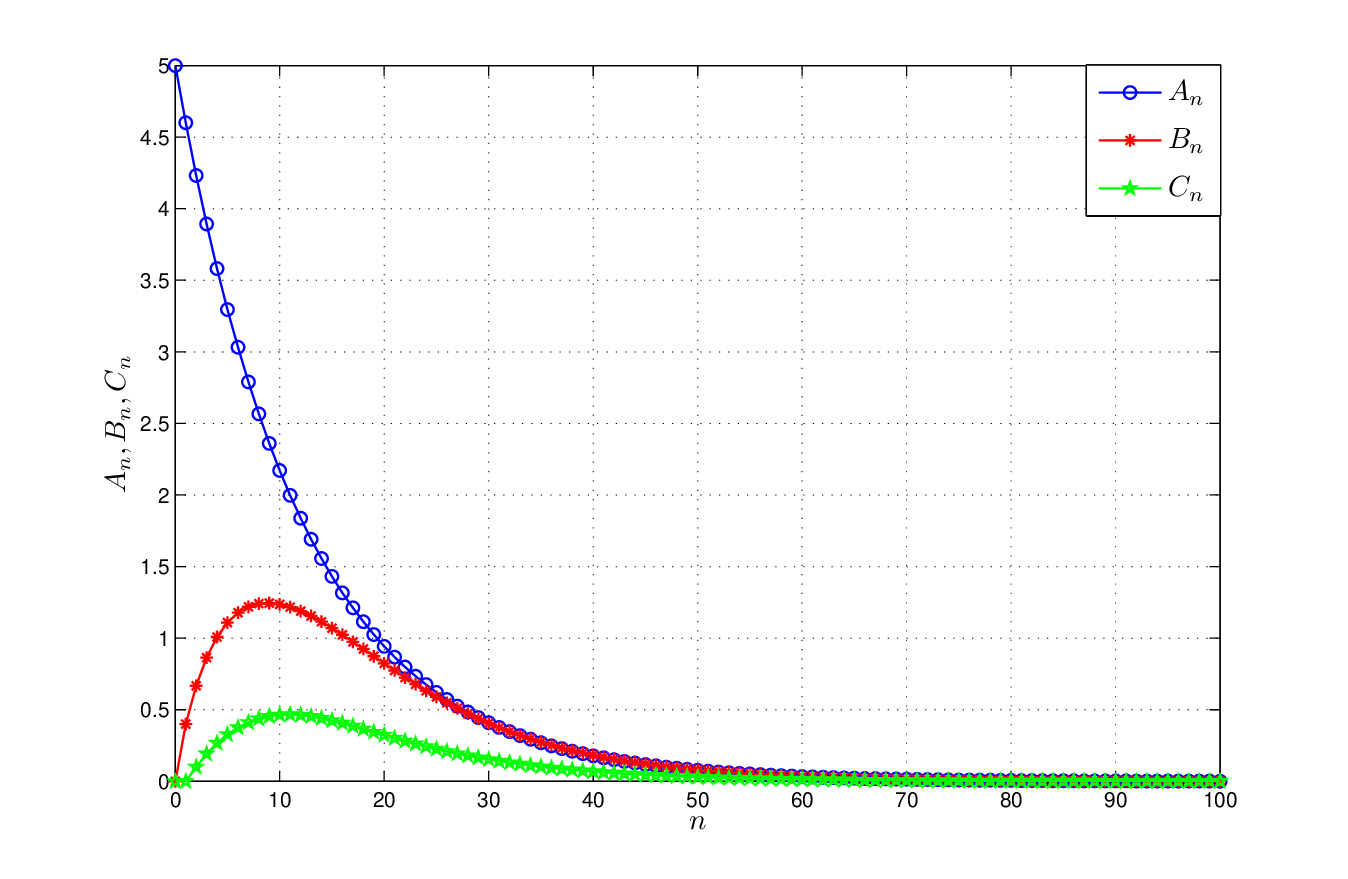}
\caption{The solutions of the discrete-time model with the initial data $\big(A_0,\,B_0,\,C_0\big) = (5,\,0,\,0)$.}\label{Fig:9}
\end{figure}

\begin{figure}[H]
\centering
\includegraphics[height=9.5cm,width=16cm]{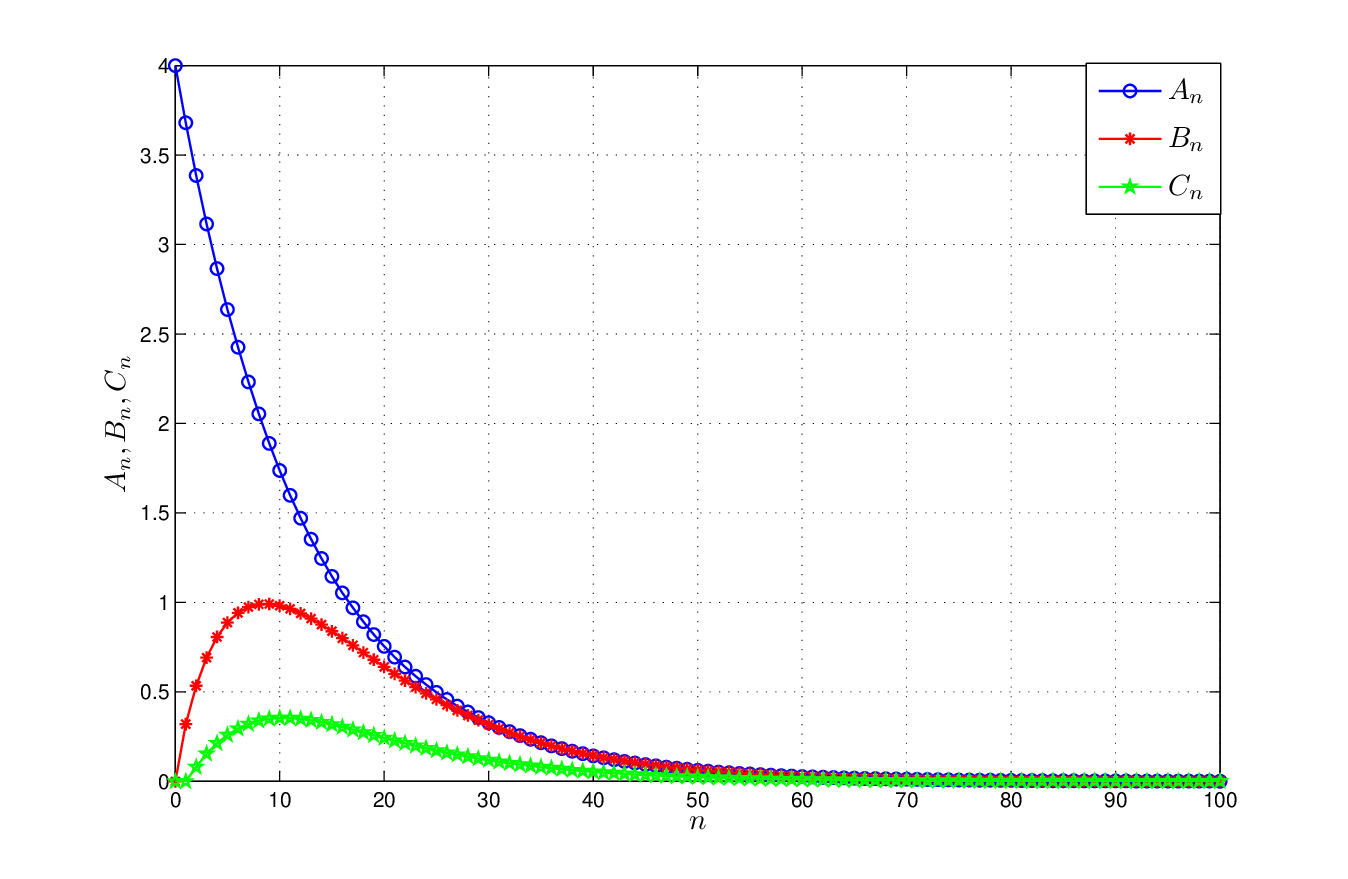}
\caption{The solutions of the discrete-time model with the initial data $\big(A_0,\,B_0,\,C_0\big) = (4,\,0,\,0)$.}\label{Fig:10}
\end{figure}

\begin{figure}[H]
\centering
\includegraphics[height=9.5cm,width=16cm]{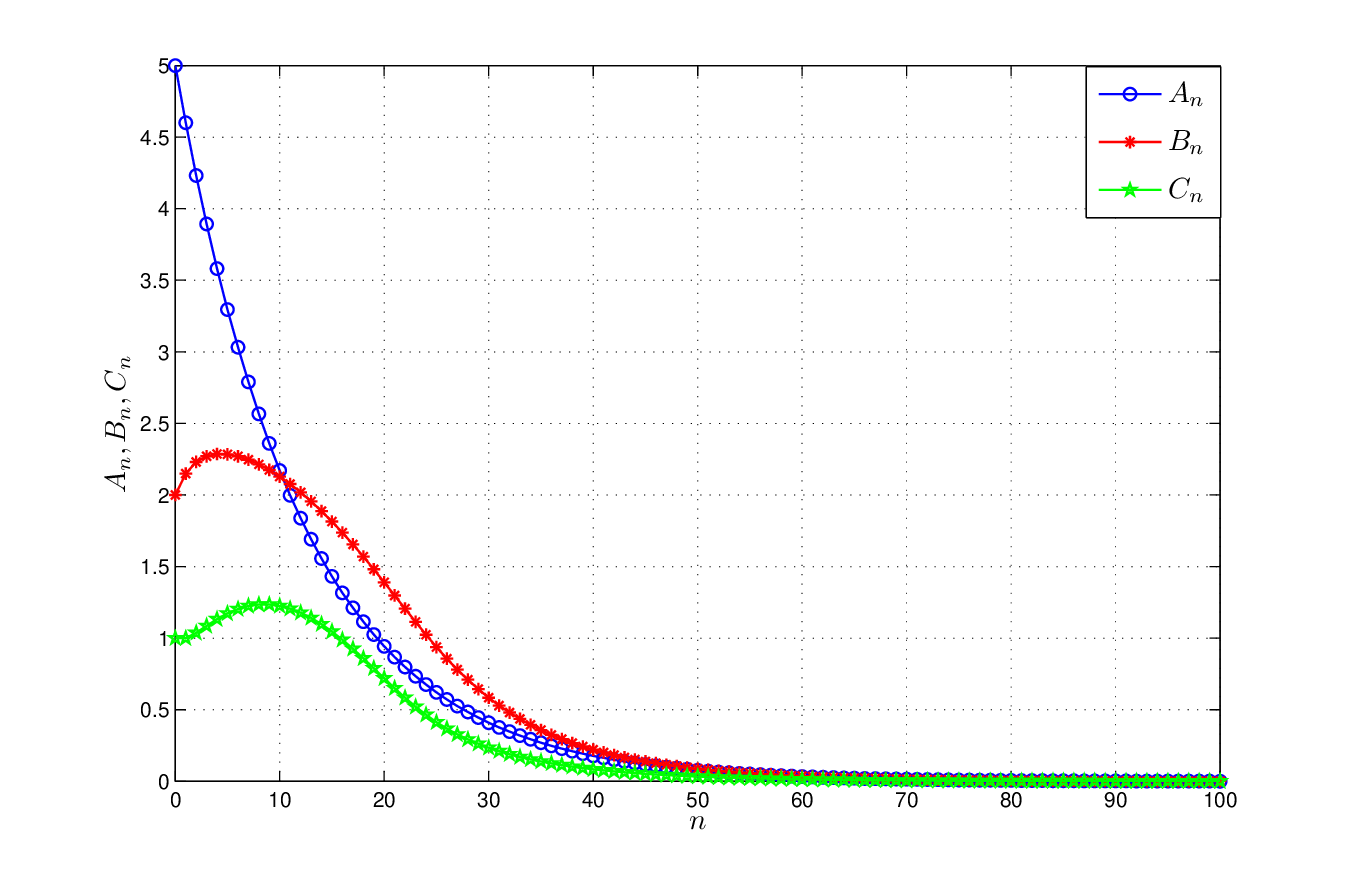}
\caption{The solutions of the discrete-time model with the initial data $\big(A_0,\,B_0,\,C_0\big) = (5,\,2,\,1)$.}\label{Fig:11}
\end{figure}

\begin{figure}[H]
\centering
\includegraphics[height=9.5cm,width=16cm]{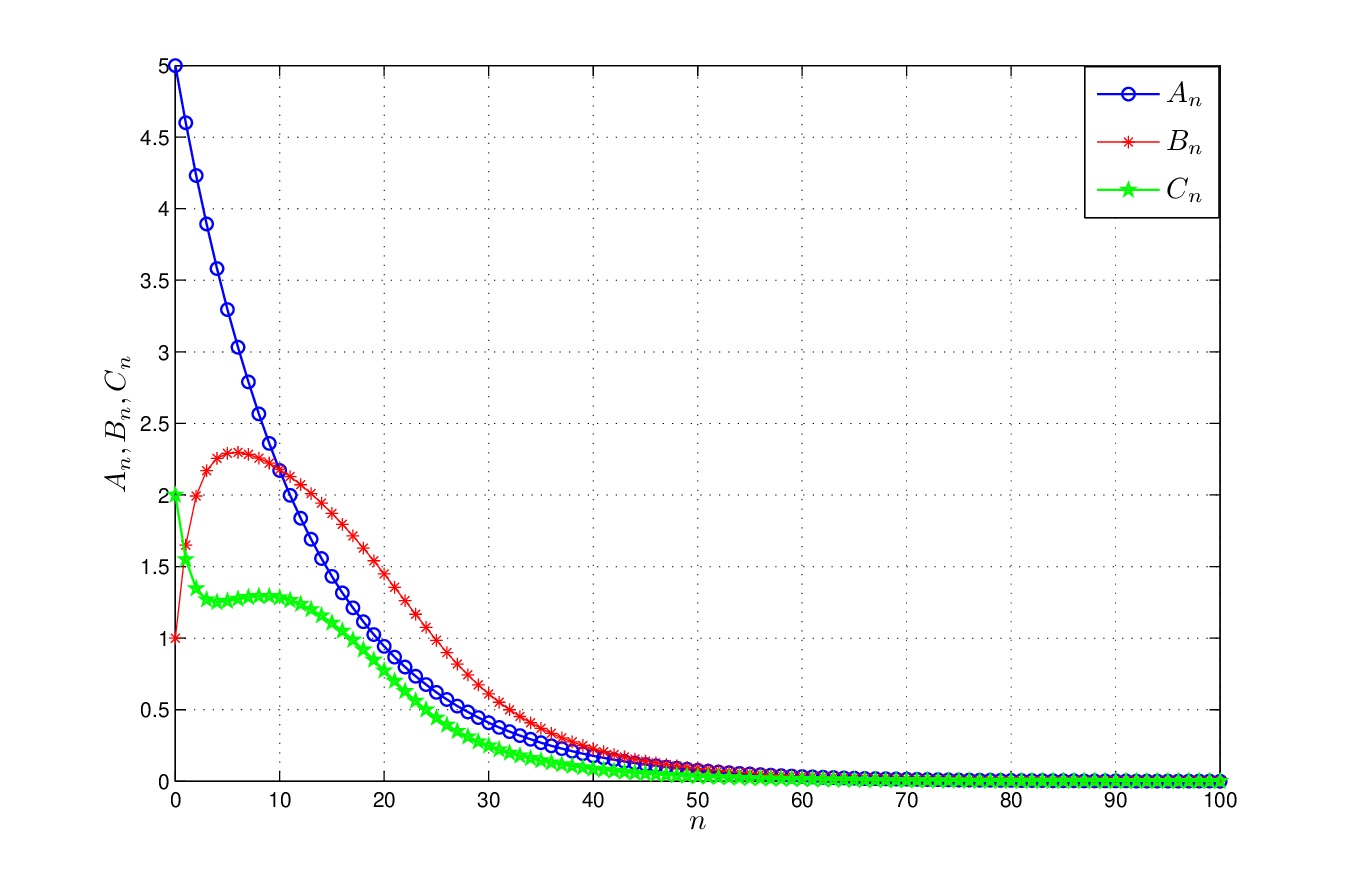}
\caption{The solutions of the discrete-time model with the initial data $\big(A_0,\,B_0,\,C_0\big) = (5,\,1,\,2)$.}\label{Fig:12}
\end{figure}

}

\section{Concluding Remarks and Discussions}\label{Sec4}
As the main conclusion of this work, we have introduced a generalized three-compartment system modeling ethanol metabolism in the human body, which extends a model recently constructed in \cite{Wacker1}. This model replaced the Michaelis-Menten mechanism of the liver's ethanol metabolism rate by a general class of nonlinear rate functions.  This provides greater modeling flexibility and allows the model to capture a wider variety of hepatic ethanol metabolism dynamics. The qualitative dynamics of the proposed ethanol metabolism model has been analyzed rigorously, including the positivity and boundedness of the solutions and the global asymptotic stability of the unique equilibrium point by using a quadratic Lyapunov function.

Second, we have formulated a discrete-time counterpart of the proposed continuous-time model and investigated its dynamical properties. We have shown that, under an appropriate condition on the time step size, the discrete-time model faithfully reproduces the qualitative dynamical behavior of the corresponding continuous-time system.

 Furthermore, a series of numerical experiments employing several ethanol metabolism rate functions has been conducted to support the theoretical findings.

In the near future, we plan to apply the proposed model to real-world situations. In particular, parameter estimation and identification problems will be of interest.\\
\textbf{Availability of supporting data:} The data supporting the findings of this study are available within the article [and/or] its supplementary materials.\\
\textbf{Conflicts of Interest:} The author declares no conflicts of interest to disclose.\\
\textbf{Authors' contributions:} \textbf{Manh Tuan Hoang and Thi Kim Quy Ngo:} 
Writing review \& editing, Writing original draft, Visualization, Validation, Supervision, Software, Resources, Project administration, Methodology, Investigation,
Formal analysis, Data curation, Conceptualization, Funding acquisition.\\
\textbf{Benjamin Wacker:} 
Writing review \& editing, Visualization, Validation, Software, Resources, Project administration, Methodology, Investigation,
Formal analysis, Data curation, Conceptualization, Funding acquisition.\\
\textbf{Funding information:} Not available.


\bibliographystyle{amsalpha}

\end{document}